\documentclass[secthm,seceqn,amsthm,ussrhead,12pt]{amsart}
\usepackage{amsmath,latexsym}
\usepackage[english]{babel}
\usepackage[psamsfonts]{amssymb}
\usepackage{times}
\usepackage{cite}
\usepackage{color}
\usepackage[mathcal]{euscript}
\numberwithin{equation}{section} \textwidth=15.5cm
\newtheorem{theorem}{Theorem}[section]

\newtheorem{corollary}[theorem]{Corollary}

\newtheorem{lemma}[theorem]{Lemma}

\theoremstyle{definition}

\usepackage[colorlinks, bookmarks=true]{hyperref}
\usepackage{color,graphicx,shortvrb}

\usepackage{enumerate}
\usepackage{amsmath}
\usepackage{amssymb}

\usepackage[latin 1]{inputenc}

\begin{document}

\numberwithin{equation}{section}

\title[$2$-local automorphisms on finite-dimensional
Lie algebras]{$2$-local automorphisms on finite-dimensional Lie
algebras}

\author[Ayupov]{Shavkat Ayupov}
\email{sh$_{-}$ayupov@mail.ru}
\address{Dormon yoli 29, Institute of
 Mathematics,  National University of
Uzbekistan,  100125  Tashkent,   Uzbekistan}

\author[Kudaybergenov]{Karimbergen Kudaybergenov}
\email{karim2006@mail.ru}
\address{Ch. Abdirov 1, Department of Mathematics, Karakalpak State University, Nukus 230113, Uzbekistan}



\date{}
\maketitle

\begin{abstract} We prove that every $2$-local
automorphism on a finite-dimensional  semi-simple  Lie algebra
$\mathcal{L}$ over an algebraically closed field of characteristic
zero is an  automorphism. We also show that each finite-dimensional
nilpotent Lie algebra $\mathcal{L}$ with $\dim \mathcal{L}\geq 2$
admits a $2$-local automorphism which is not an automorphism.

\\


\end{abstract}

\maketitle
\thispagestyle{empty}

\section{Introduction}\label{sec:intro}

In last decades a series  of papers have been devoted to
 mappings which are close to automorphism and derivation of
associative algebras (especially of operator algebras and C*-algebras).
Namely, the problems of description of so called
 local automorphisms (respectively, local derivations) and 2-local
automorphisms (respectively, 2-local derivations) has been considered.
Later similar problems were extended for non associative algebras,
in particular, for Lie algebras case. The present paper is devoted to
 study of 2-local automorphisms of finite dimensional Lie algebras.

Let
$\mathcal{A}$ be an associative  algebra. Recall that a linear mapping
$\Phi$ of $\mathcal{A}$ into itself is called a local automorphism
(respectively, a local derivation) if for every $x\in \mathcal{A}$
there exists an automorphism (respectively, a derivation) $\Phi_x$
of $\mathcal{A},$ depending on $x,$ such that $\Phi_x(x)=\Phi(x).$
These notions were introduced and investigated independently by
Kadison~\cite{Kadison90} and Larson and Sourour~\cite{Larson90}.
Later, in 1997, P.~\v{S}emrl~\cite{Semrl97}  introduced the concepts of
2-local automorphisms and 2-local derivations. A map
$\Phi:\mathcal{A} \rightarrow \mathcal{A}$ (not linear in general)
is called a 2-local automorphism (respectively, a 2-local
derivation) if for every $x, y\in \mathcal{A},$ there exists an
automorphism (respectively, a derivation) $\Phi_{x,y}:\mathcal{A}
\rightarrow \mathcal{A}$ (depending on $x, y.$) such that $\Phi_{x.y}(x)=\Phi(x),$
$\Phi_{x,y}(y)=\Phi(y).$   In \cite{Semrl97}, P.~\v{S}emrl
described 2-local derivations and 2-local automorphisms on the
algebra $B(H)$ of all bounded linear operators on the
infinite-dimensional separable Hilbert space $H.$ Namely, he has proved that
 every 2-local automorphism (respectively, 2-local derivation) on $B(H)$ is an automorphism
 (respectively, a derivation).
 A similar result for finite-dimensional case appeared later
in~\cite{Kim04}. Further, in~\cite{AK}, a new techniques was introduced
to prove the same result for an arbitrary Hilbert space $H$ (no separability is
assumed).

Afterwards the above considerations  give  arise to similar questions in von Neumann algebras framework.
First positive results have been obtained in ~\cite{AKNA} and ~\cite{AA2}
finite and semi-finite von Neumann algebras respectively, by showing that all
2-local derivations on these algebras are derivations. Finally, in~~\cite{AK14},
the same result was obtained for purely infinite von Neumann algebras. This completed the
solution of the above problem for arbitrary von Neumann algebras.

It is natural to study corresponding analogues of these problems for
automorphisms or derivations of non-associative algebras. We shall consider
 the case of finite-dimensional  Lie algebras.

Let$\mathcal{L}$ be a Lie algebra. A derivation (respectively, an
automorphism) $\Phi$ of a Lie algebra $\mathcal{L}$ is a linear
(respectively, an invertible linear) map $\Phi:\mathcal{L}
\rightarrow \mathcal{L}$ which satisfies the condition  $\Phi([x, y])
=[\Phi(x), y] +[x, \Phi(y)]$ (respectively, $\Phi([x, y])
=[\Phi(x), \Phi(y)])$ for all $x, y\in \mathcal{L}.$ The set of
all automorphisms of a Lie algebra $\mathcal{L}$  is denoted by
$\mbox{Aut} \mathcal{L}.$

The notions of
a local derivation (respectively, a local automorphism) and a 2-local derivation
(respectively, a 2-local automorphism) for Lie algebras are defined as above,
similar to the associative case. Every derivation (respectively,
automorphism) of a Lie algebra $\mathcal{L}$ is a local derivation
(respectively, local automorphism) and a 2-local derivation
(respectively, 2-local automorphism). For a given Lie algebra
$\mathcal{L},$ the main problem concerning these notions is to prove that
they automatically become a  derivation (respectively, an automorphism)
or to give examples of local and 2-local derivations or automorphisms of $\mathcal{L},$
which are not  derivations or automorphisms, respectively.

 For a finite-dimensional semi-simple Lie algebra
$\mathcal{L}$ over an algebraically closed field of characteristic
zero, the derivations and automorphisms of $\mathcal{L}$ are
completely described in \cite{Humphreys}.

Recently in \cite{AK2016}
 we proved that every local derivation on semi-simple Lie algebras are
derivations and gave examples of nilpotent finite-dimensional Lie
algebras with local derivations which are not derivations.

Earlier in~\cite{AKR15} the authors have proved that every 2-local derivation on a
semi-simple  Lie algebra $\mathcal{L}$ is a derivation, and showed that each
finite-dimension nilpotent Lie algebra, with dimension larger than
two,  admits a 2-local derivation which is not a derivation.

In~\cite{Chen},  Chen and Wang initiated study of  2-local
automorphisms of  finite-dimensional Lie algebras. They have proved that
if $\mathcal{L}$ is a simple Lie algebra which belongs to one of the  types
 $A_l\, (l\geq 1), D_l\, (l\geq4),$ or  $E_k\,
(k=6, 7, 8)$ over an algebraically closed field of characteristic
zero, then every 2-local automorphism of $\mathcal{L},$ is an automorphism.

In the present paper we generalize the result of  ~\cite{Chen} and
prove that every $2$-local automorphisms of finite-dimensional semi-simple Lie algebras over
an algebraically closed field of characteristic zero is an automorphism. Moreover, we show
that each  finite-dimensional nilpotent Lie algebra admits
a 2-local automorphism which is not an automorphism.

\section{$2$-local automorphisms on finite-dimensional semi-simple  Lie algebras}
\label{lie}

In this paper the notations concerning Lie algebras mainly follow
 \cite{Humphreys}.

All algebras and vector spaces are considered  over an
algebraically closed field  $\mathbb{F}$ of characteristic zero.

The main result of this section is given as follows.

\begin{theorem}\label{mainautho}
Let $\mathcal{L}$ be an arbitrary  finite-dimensional semi-simple
Lie algebra over  an algebraically closed field  $\mathbb{F}$ of
characteristic zero. Then any {\rm(}not necessarily linear{\rm)}
$2$-local automorphism $T: \mathcal{L}\rightarrow \mathcal{L}$ is
an automorphism.
\end{theorem}

.

Let $\mathcal{L}$ be a Lie algebra. The \textit{center} of $\mathcal{L}$ is denoted by $Z(\mathcal{L}):$
$$
Z(\mathcal{L})=\{x\in \mathcal{L}: [x,y]=0,\,\forall\,  y\in
\mathcal{L}\}.
$$

 A Lie algebra $\mathcal{L}$ is said
to be \textit{solvable} if $\mathcal{L}^{(k)}=\{0\}$ for some
integer $k,$ where $\mathcal{L}^{(0)}=\mathcal{L},$
$\mathcal{L}^{(k)}=[\mathcal{L}^{(k-1)}, \mathcal{L}^{(k-1)}],\,
k\geq1.$ Any Lie algebra $\mathcal{L}$ contains a unique maximal
solvable ideal, called the radical of $\mathcal{L}$ and denoted by
$\mbox{Rad} \mathcal{L}.$ A non trivial Lie algebra $\mathcal{L}$
is called \textit{semi-simple} if $\mbox{Rad} \mathcal{L}=0.$ That
is equivalent to requiring that $\mathcal{L}$ have no nonzero
abelian ideals.

Given a vector space $V,$ let $\mathfrak{gl}(V)$ denote the Lie algebra
of all linear endomorphisms
of $V.$ A representation of a Lie algebra $\mathcal{L}$ on $V$ is
a Lie algebra homomorphism $\rho:\mathcal{L}\rightarrow \mathfrak{gl}(V).$
For example, $\mbox{ad}:\mathcal{L}\rightarrow \mathfrak{gl}(\mathcal{L})$
given by $\mbox{ad}(x)(y)=[x,y]$ is a representation of
$\mathcal{L}$ on the vector space $\mathcal{L}$ called the adjoint
representation. If $V$ is a finite-dimensional vector space then
 the representation $\rho$ is said to be finite-dimensional.

 Let $\mathcal{L}$ be a Lie algebra and let $\rho:\mathcal{L}\rightarrow \mathfrak{gl}(V)$ be a
finite-dimensional representation of $\mathcal{L}.$ Then the map
$\tau:\mathcal{L}\times \mathcal{L}\rightarrow \mathbb{F}$ defined
by
$$
\tau(x,y)=\mbox{tr}(\rho(x)\rho(y))
$$
is a symmetric bilinear form on $\mathcal{L}$  called the
\textit{trace form} on $\mathcal{L}$ relative to $\rho,$ where
$\mathrm{tr}$ denotes the trace of a linear operator. In
particular, for $V=\mathcal{L}$ and $\rho=\mathrm{ad}$ the
corresponding trace form is called the \textit{Killing form} and
it is denoted by $\langle\cdot,\cdot\rangle.$ The Killing form has
the following $\mbox{Aut}(\mathcal{L})$-invariance  property (see
\cite[P. 231]{SW}), which is an important tool  in the study of $2$-local
automorphisms of semi-simple Lie algebras:
\begin{center}
$\langle \Phi(x), \Phi(y)\rangle=\langle x, y\rangle$ for all
$x,y\in \mathcal{L}.$
\end{center}

Another importance of the Killing form is the following property.
A Lie algebra $\mathcal{L}$ is semi-simple if and only if its
Killing form is non degenerate, i.e. $\langle x, y\rangle=0$ for
all $y\in \mathcal{L}$ implies that $x=0.$

 A \textit{Cartan subalgebra} $\mathcal{H}$ of   a
 semi-simple Lie algebra $\mathcal{L}$ is a nilpotent subalgebra
which coincides with its centralizer: $C(\mathcal{H}) = \{x\in
\mathcal{L}: [x, h] = 0,\, \forall h\in \mathcal{H}\} =
\mathcal{H}.$

A Cartan subalgebra $\mathcal{H}$ of a  finite-dimensional
semi-simple Lie algebra $\mathcal{L}$ is \textit{abelian}, i.e.
$[x,y]=0$ for all $x,y\in \mathcal{H}.$

From now on in this section let  $\mathcal{L}$  be a
finite-dimensional simple Lie algebra of the rank
$l,$   and let $\mathcal{H}$ be a fixed standard Cartan subalgebra of
$\mathcal{L}$ with  the corresponding root
system  $R\subset \mathcal{H}$. Denote by  $\triangle =\{\alpha_1, \alpha_2, \cdots, \alpha_l\}$
a fixed base of $R,$ and by $R^+$  the set of corresponding positive root
system of $\mathcal{L}$ relative to $R.$  For $\alpha\in R,$ let
$$
L_\alpha=\{x \in  L: [h, x] = \alpha(h)x, \, \forall\, h\in
\mathcal{H}\}
$$
be the one-dimensional root space relative to $\alpha.$ For each
$\alpha\in R^+,$  let $e_\alpha$ be a non-zero element of
$\mathcal{L}_\alpha,$ then there is a unique element
$e_{-\alpha}\in \mathcal{L}_{-\alpha}$ such that the linear span of the elements $e_\alpha,
e_{-\alpha},$ and $ h_\alpha=[e_\alpha, e_{-\alpha}]$ is a
three-dimensional simple subalgebra of $\mathcal{L}$ isomorphic to the simple Lie algebra
$\mathfrak{sl}(2, \mathbb{F}).$  The set $\{h_\alpha, e_\alpha,
e_{-\alpha}: \alpha\in R^+\}$ forms a basis of $\mathcal{L}.$

By \cite[Lemma~2.2]{Wang}, there exists an element $d\in
\mathcal{H}$ such that $\alpha(d) \neq \beta(d)$ for every
$\alpha, \beta\in R, \alpha\neq \beta.$  Such elements $d$ are
called \textit{strongly regular} elements of $\mathcal{L}.$ Again
by \cite[Lemma~2.2]{Wang}, every strongly regular element is a
\textit{regular semi-simple element}, i.e.
$$
\{x\in \mathcal{L}: [d, x]=0\}=\mathcal{H}.
$$
Choose a fixed strongly regular element $d\in \mathcal{H}.$ Also
set
\begin{equation}\label{qqq}
q =\sum\limits_{\alpha\in R}e_\alpha.
\end{equation}

\begin{lemma}\label{twoauto} Let $\Phi$ be an automorphism on \(\mathcal{L}\) such that \(\Phi(d)=d.\)
Then
\begin{enumerate}
\item[(i)] For any $h \in \mathcal{H},$  $\Phi(h)=h;$
\item[(ii)] For any $\alpha\in R,$  there exists a nonzero  $c_\alpha\in \mathbb{F}$ such
that $\Phi(e_\alpha)=c_\alpha e_\alpha.$
\end{enumerate}
\end{lemma}

\begin{proof}
For \(\alpha\in R\) we have
\begin{eqnarray*}
\left[d, \Phi(e_\alpha)\right] & = & \left[\Phi(d),
\Phi(e_\alpha)\right]=\Phi\left([d,
e_\alpha]\right)=\Phi(\alpha(d)e_\alpha)=\alpha(d)\Phi(e_\alpha),
\end{eqnarray*}
i.e.
\begin{eqnarray}\label{dalpha} \left[d, \Phi(e_\alpha)\right]
=\alpha(d)\Phi(e_\alpha).
\end{eqnarray}
Let us represent the element \(\Phi(e_\alpha)\) in the form
\begin{eqnarray*}
\Phi(e_\alpha) & = & h_\alpha+\sum\limits_{\beta\in R} c_\beta
e_\beta,
\end{eqnarray*}
where \(h_\alpha\in \mathcal{H},\) \(c_\beta\in \mathbb{F}.\)
Multiplying both sides of the last equality  by \(\alpha(d)\) and
taking into account \eqref{dalpha}, we have
\begin{eqnarray}\label{dbeta}
\sum\limits_{\beta\in R} c_\beta
\beta(d)e_\beta=\alpha(d)h_\alpha+\sum\limits_{\beta\in R} c_\beta
\alpha(d)e_\beta.
\end{eqnarray}
By properties of strongly regular elements, \(\alpha(d)\neq 0,\)
and \(\beta(d)\neq \alpha(d)\) for any \(\beta\neq\alpha.\) Thus
by the equality \eqref{dbeta}, \(h_\alpha=0\) and \(c_\beta=0\)
for any \(\beta\neq\alpha,\) \(\beta\in R.\)   Therefore
\(\Phi(e_\alpha)=c_\alpha e_\alpha.\) Since \(\Phi\) is a
bijection, then \(\Phi(e_\alpha) \neq 0,\) and so
\(c_\alpha\neq0.\) Thus (ii) holds.

Take an arbitrary element  \(h\in \mathcal{H}.\) We have
\begin{eqnarray*}
\left[d, \Phi(h)\right] & = & \left[\Phi(d),
\Phi(h)\right]=\Phi\left([d, h]\right)=\Phi(0)=0.
\end{eqnarray*}
Since \(d\) is a regularly semi-simple element, it follows that
\(\Phi(h)\in \mathcal{H}.\)

For every \(\alpha\in R\) we have
\begin{eqnarray*}
\left[\Phi(h), \Phi(e_\alpha)\right] & = & \Phi\left([h,
e_\alpha]\right)=\Phi(\alpha(h)e_\alpha)=\alpha(h)\Phi(e_\alpha)=\alpha(h)c_\alpha
e_\alpha.
\end{eqnarray*}
i.e.
\begin{eqnarray*}
\left[\Phi(h), \Phi(e_\alpha)\right] & = & \alpha(h)c_\alpha
e_\alpha.
\end{eqnarray*}

On the other hand, taking into account \(\Phi(h)\in \mathcal{H},\)
we have
\begin{eqnarray*}
\left[\Phi(h), \Phi(e_\alpha)\right] & = & \left[\Phi(h), c_\alpha
e_\alpha\right]=c_\alpha \alpha\left(\Phi(h)\right)e_\alpha.
\end{eqnarray*}
Thus \(c_\alpha \alpha(h)=c_\alpha \alpha\left(\Phi(h)\right),\)
which implies that \(\alpha(h)=\alpha\left(\Phi(h)\right)\) for
any \(\alpha\in R,\)  Therefore \(\Phi(h) = h,\) i.e., (i) holds.
The proof is complete.
\end{proof}

The following  assertion has been proved in \cite[Lemma~3.3]{Chen}
for the case of simple algebras of types $A_l\, (l\geq 1), D_l\,
(l\geq4), E_k\, (k=6, 7, 8).$ From Lemma~\ref{twoauto} we get the
following generalization for an arbitrary semi-simple Lie algebra.

\begin{lemma}\label{two} Let $T$ be a 2-local automorphism \(\mathcal{L}\) such that \(T(d)=d.\)
Then
\begin{enumerate}
\item $T(h)=h$ for all $h \in \mathcal{H};$
\item for every  $\alpha\in R,$  there exists a nonzero  $c_\alpha\in \mathbb{F}$ such
that $T(e_\alpha)=c_\alpha e_\alpha.$
\end{enumerate}
\end{lemma}

\begin{proof} For \(h\in \mathcal{H}\) take an automorphism
\(\Phi_ {h,d}\) such that
\begin{center}
\(\Phi_ {h,d}(h)=T(h)\) and \(\Phi_ {h,d}(d)=T(d).\)
\end{center}
Since \(\Phi_ {h,d}(d)=T(d)=d,\)  Lemma~\ref{twoauto} implies
that \(\Phi_ {h,d}(h)=h,\) and therefore \(T(h)=h.\)

By a similar way we can check (2). The proof is complete.
\end{proof}

\begin{lemma}\label{addit} Let $T$ be a $2$-local
automorphism on a finite-dimensional semi-simple Lie algebra
$\mathcal{L}$ such that  $T(d)=d.$  Then $T$ is linear.
\end{lemma}

\begin{proof}
Firstly we show that
\begin{equation}\label{localkilling}
\langle T(x), T(y)\rangle=\langle x, y\rangle
\end{equation}
for all $x,y\in \mathcal{L}.$ Indeed, taking into account
$\mbox{Aut}(\mathcal{L})$-invariance of the Killing form we obtain
that
\begin{eqnarray*}
\langle T(x), T(y)\rangle & = & \langle \Phi_{x,y}(x),
\Phi_{x,y}(y)\rangle=\langle x, y\rangle.
\end{eqnarray*}

Now let $x,y,z\in \mathcal{L}$ be arbitrary elements. Using
equality~\eqref{localkilling} we obtain that
\begin{eqnarray*}
\langle T(x+y), T(z)\rangle & = & \langle x+y, z\rangle= \langle
x, z\rangle+ \langle y, z\rangle=  \\
&  = & \langle T(x), T(z)\rangle+\langle T(y), T(z)\rangle=
\langle T(x)+T(y), T(z)\rangle,
\end{eqnarray*}
i.e.
$$
\langle T(x+y), T(z)\rangle= \langle T(x)+T(y), T(z)\rangle.
$$
or
$$
\langle T(x+y)-T(x)-T(y), T(z)\rangle=0.
$$
Putting in the last equality $z=h\in \mathcal{H}$ and
$z=e_\alpha,$ where $\alpha\in R,$ by Lemma~\ref{two} we obtain
that
$$
\langle T(x+y)-T(x)-T(y), h\rangle=0
$$
and
$$
\langle T(x+y)-T(x)-T(y), e_\alpha\rangle=0
$$
for all $h\in \mathcal{H}$ and $\alpha\in R.$ Taking into account
bilinearity of the Killing form we have that
$$
\langle T(x+y)-T(x)-T(y), w\rangle=0
$$
for all $w\in \mathcal{L}.$ Since the Killing form $\langle\cdot,
\cdot\rangle$ is non-degenerate, the last equality implies that $
T(x+y)=T(x)+T(y)$ for all $x, y\in \mathcal{L}.$

Finally,
$$
T(\lambda x)=\Phi_{\lambda x, x}(\lambda x)=\lambda \Phi_{\lambda
x, x}(x)=\lambda T(x).
$$
So, $T$ is linear.  The proof is complete.
\end{proof}

\begin{proof}[\textit{Proof of Theorem~$\ref{mainautho}$}]
Let $T:\mathcal{L}\rightarrow \mathcal{L}$ be a $2$-local
automorphism  and suppose that $d\in \mathcal{H}$ is a strongly  regular element
and $q$ is the element defined by~\eqref{qqq}. Take an
automorphism $\Phi_{d, q}$ on $\mathcal{L}$ (depending on $d$ and
$q$) such that
\begin{center}
$\Phi_{d, q}(d)=T(d)$ and $\Phi_{d, q}(q)=T(q).$
\end{center}
Set $T_1=\Phi^{-1}_{d, q}\circ T.$ Then \(T_1\) is a
\(2\)-automorphism such that
\begin{center}
$T_1(d)=d$ and $T_1(q)=q.$
\end{center}
By Lemma~\ref{addit}, $T_1$ is linear.

Taking into account Lemma~\ref{two} we have
\begin{eqnarray*}
T_1(q) & = & T_1\left(\sum\limits_{\alpha\in R} e_\alpha\right)=
\sum\limits_{\alpha\in R}
T_1\left(e_\alpha\right)=\sum\limits_{\alpha\in R}  c_\alpha
e_\alpha.
\end{eqnarray*}
On the other hand
\begin{eqnarray*}
T_1(q) & = & q =\sum\limits_{\alpha\in R} e_\alpha.
\end{eqnarray*}
Comparing the last two equalities we obtain that $c_\alpha=1$ for
all $\alpha\in R.$ So,
\begin{center}
$T_1(h)=h$ and $T_1(e_\alpha)=e_\alpha$
\end{center}
for all $h\in \mathcal{H}$ and $\alpha\in R.$ Finally, in view
of the linearity of $T_1,$ the last two equalities imply
that $T_1(x)=x$ for all $x\in \mathcal{L}.$ Thus $T=\Phi_{d, q}$
is an automorphism.  The proof is complete.
\end{proof}

\section{$2$-local derivations on nilpotent Lie algebras}

In this section we give examples of $2$-local automorphisms on
nilpotent Lie algebras which are not automorphisms.

A Lie algebra $\mathcal{L}$ is called \textit{nilpotent} if
$\mathcal{L}^k=\{0\}$ for some $k \in \mathbb{N},$  where
$\mathcal{L}^0=\mathcal{L},$ $\mathcal{L}^k=[\mathcal{L}^{k-1},
\mathcal{L}],\, k\geq1.$

Let $D$ be a nilpotent derivation on $\mathcal{L},$  i.e.
\(D^n=0\) for some \(n\in \mathbb{N}.\) Then
$$
\Phi(x)=(\exp D)(x)=x+D(x)+\frac{D^2(x)}{2!}+\cdots
+\frac{D^{n-1}(x)}{(n-1)!}
$$
is an automorphism (see, e.g. \cite[P. 9]{Jacob}).

\begin{theorem}\label{pure}
Let $\mathcal{L}$ be a $n$-dimensional Lie algebra with $n\geq2.$
Suppose that
\begin{itemize}
\item[(i)] $\dim[\mathcal{L},\mathcal{L}]\leq n-2;$
\item[(ii)] the
intersection  $Z(\mathcal{L})\cap [\mathcal{L}, \mathcal{L}]$ is
non trivial.
\end{itemize}
Then $\mathcal{L}$ admits a $2$-local automorphism which is not an
automorphism.
\end{theorem}

\begin{proof}  Let us consider a decomposition of  $\mathcal{L}$ in the following
form
$$
\mathcal{L}=[\mathcal{L},\mathcal{L}]\oplus V.
$$
Due to the assumption $\dim[\mathcal{L},\mathcal{L}]\leq n-2,$ we
have $\dim V=k\geq 2.$ Let $\{e_1,\ldots, e_k\}$ be a basis of $V$
and let \(f\) be a homogeneous non additive function  on
$\mathbb{F}^2.$

According to the assumptions (ii) of the theorem there exists a
non zero central element in $[\mathcal{L}, \mathcal{L}].$  Let us
fix it as $z\in Z(\mathcal{L}).$ By the proof of Theorem~~3.1 from
\cite{AKR15} it follows that an operator $T$ on $\mathcal{L}$
defined as
\begin{center}
$T(x)=f(\lambda_1,\lambda_2)z,\,\,$ for
$\,\,x=x_1+\sum\limits_{i=1}^k\lambda_i e_i \in \mathcal{L},$
\end{center}
where $\lambda_i\in \mathbb{F},$ $i=1,\ldots,k,$ $x_1\in
[\mathcal{L},\mathcal{L}],$  is a 2-local derivation which is not
a  derivation.

Set
$$
\Delta(x)=x+T(x).
$$
Let us show that $\Delta$ is a $2$-local automorphism.  Again by
the proof  of Theorem~~3.1 from \cite{AKR15} it follows that for
every pair \(x, y\in \mathcal{L}\)  there exist \(a, b\in
\mathbb{F}\) such that the operator \(D\) defined as
\begin{center}
$D(x)=(a\lambda_1+b\lambda_2)z,\,\,$ for
$\,\,x=x_1+\sum\limits_{i=1}^k \lambda_i e_i \in \mathcal{L},$
\end{center}
is a derivation and
$$
T(x)=D(x),\,\, T(y)=D(y).
$$
Since \(D(\mathcal{L})=\{\lambda z: \lambda\in \mathbb{F}\}\) and
\(z\in Z(\mathcal{L})\cap [\mathcal{L}, \mathcal{L}],\) it follows
that \(D\) is a nilpotent derivation of order 2, i.e. \(D^2=0.\)
Therefore,
$$
\Phi(x)=(\exp D)(x)=x+D(x)
$$
is an automorphism.  Since
$$
\Delta(x)=\Phi(x),\,\, \Delta(y)=\Phi(y),
$$
$\Delta$ is a $2$-local automorphism, as required. The proof is
complete.
\end{proof}

Let $\mathcal{L}$ be an $n$-dimensional nilpotent Lie algebra with
$n\geq 2.$ Then $[\mathcal{L},\mathcal{L}]\neq\mathcal{L},$
otherwise, $\mathcal{L}^k=\mathcal{L}\neq\{0\}$ for all $k\geq 1.$

Suppose that $\dim[\mathcal{L},\mathcal{L}]= n-1.$ Then direct
computations show that
$$
\{0\}\neq\mathcal{L}^1=[\mathcal{L}, \mathcal{L}]
=[[\mathcal{L},\mathcal{L}],\mathcal{L}]=\mathcal{L}^2,
$$
and therefore $\mathcal{L}^k=\mathcal{L}^{k+1}\neq\{0\}$ for all
$k\in \mathbb{N}.$ This contradicts the nilpotency of
$\mathcal{L}.$ So $\dim[\mathcal{L},\mathcal{L}]\leq  n-2.$

Note that for any nilpotent Lie algebra \(Z(\mathcal{L})\cap
\left[\mathcal{L}, \mathcal{L}\right]\neq \{0\},\) because
\(\{0\}\neq \mathcal{L}^{k-1}\subseteq Z(\mathcal{L})\cap
\left[\mathcal{L}, \mathcal{L}\right],\) where
\(\mathcal{L}^{k-1}\neq \{0\}\) and \(\mathcal{L}^{k}= \{0\}.\)
So, Theorem~\ref{pure} implies the following result.

\begin{theorem}\label{nil}
Let $\mathcal{L}$ be a finite-dimensional nilpotent  Lie algebra
with $\dim \mathcal{L}\geq 2.$ Then $\mathcal{L}$ admits a
$2$-local automorphism which is not an automorphism.
\end{theorem}

The corollary below is an immediate consequence of
Theorem~\ref{nil}.

\begin{corollary}\label{abel}
Let $\mathcal{L}$ be a finite-dimensional abelian Lie algebra with
$\dim \mathcal{L}\geq 2.$ Then $\mathcal{L}$ admits a $2$-local
automorphism which is not an automorphism.
\end{corollary}

\end{document}